\documentclass[12pt,a4paper,makeidx]{amsart}
\usepackage[latin1]{inputenc}        
\usepackage[dvips]{graphics}
\usepackage[dvips]{graphicx}
\usepackage{mathrsfs}
\let\mathcal\mathscr
\usepackage{amssymb}
\usepackage{amsmath}
\usepackage{amsfonts}
\usepackage{latexsym}
\usepackage[T1]{fontenc}
\usepackage[latin1]{inputenc}
\usepackage{hyperref}
\usepackage[all,ps]{xy}
\RequirePackage{amsthm}
\RequirePackage{amssymb}
\usepackage[all,ps]{xy}
\usepackage{latexsym}
\usepackage{amscd}
\usepackage[latin1]{inputenc}
\usepackage[T1]{fontenc}
\usepackage{color}
\RequirePackage{amsthm}
\RequirePackage{amssymb}
\usepackage[all,ps]{xy}
\usepackage{latexsym}

\def\N{{\bf N}} 
\def\P{{\bf P}}

\def\R{{\bf R}}
\def\Q{{\bf Q}}

\def\Alb{\mathop{\rm Alb}\nolimits}

\def\Cosupp{\mathop{\rm Cosupp}\nolimits}

\def\Supp{\mathop{\rm Supp}\nolimits}

\def\dra{\dashrightarrow}

\def\phi{\varphi}

\def\d{{\delta}}
\def\e{{\varepsilon}}

\def\m{{\mu}}

\def\r{{\rho}}
\def\D{{\Delta}}
\def\E{{\Sigma}}
\def\G{{\Gamma}}

\def\cJ{{\mathcal J}}

\def\cO{{\mathcal O}}

\def\Eff{\mathop{\rm Eff}\nolimits}

\def\nn{\mathop{\rm NNef}\nolimits}

\def\sbs{\mathop{\rm SBs}\nolimits}

\def\Nklt{\mathop{\rm {Nklt}}\nolimits}
\def\Eff{\mathop{\rm Eff}\nolimits}

\def\Bs{\mathop{\rm Bs}\nolimits}

\def\ge{\geqslant}

\newtheorem{thm}{Theorem}[section]

\newtheorem{cor}[thm]{Corollary}

\newtheorem{prop}[thm]{Proposition}
\newtheorem{lem}[thm]{Lemma}

\newtheorem{ex}[thm]{Example}

\title[Rational connectedness modulo the Non-nef locus]{Rational connectedness modulo the Non-nef locus}
\author[A. Broustet and G.~Pacienza]{Ama\"el Broustet and Gianluca Pacienza}
\date{\today}

\begin{document}
\maketitle
{\let\thefootnote\relax
\footnote{
\noindent
\textbf{A.M.S.~classification :} 14J40. 
 }}
\numberwithin{equation}{section}

%
\begin{abstract}
It is well known that a smooth projective  Fano variety  is rationally connected. Recently Zhang \cite{zhang} (and later Hacon and McKernan \cite{HM} as a special case of their work on the Shokurov RC-conjecture) proved that the same conclusion holds for a klt pair $(X,\D)$ such that  $-(K_X+\D)$ is big and nef. We prove here a natural generalization of the above result by dropping the nefness assumption. Namely 
we show that a klt pair $(X,\D)$ such that $-(K_X+\D)$ is big is rationally connected modulo the non-nef locus of $-(K_X+\D)$. This result is a consequence of a more general structure theorem for arbitrary pairs $(X,\D)$  with $-(K_X+\D)$ pseff.
\end{abstract}
%
%
%
\section{Introduction}
%
An important and nowadays classical property of smooth Fano varieties is their rational connectedness, which was established by Campana \cite{Cam1} and by Koll\'ar, Miyaoka and Mori \cite{KMM}.  
Singular varieties, and more generally pairs, with a weaker positivity property arise naturally in the Minimal Model Program, and it was conjectured  that a klt pair $(X,\D)$ with 
$-(K_X+\D)$ big and nef is rationally connected. The conjecture was proved by Zhang \cite{zhang}, and was also obtained later by Hacon and McKernan \cite{HM} as a special case of their work on the Shokurov RC-conjecture.
Even in the smooth case, as soon as the nefness hypothesis is dropped one can loose the rational connectedness, as shown by the following example (which is a natural generalization to arbitrary dimension of \cite[Example 6.4]{taka2}). 
\begin{ex}\label{ex:cone}
{\textrm Let ${\bf P}^N$ be a hyperplane of ${\bf P}^{N+1}$, $Z$ a smooth hypersurface of degree $N+1$ in ${\bf P}^N$ and $C_Z \subset {\bf P}^{N+1}$ a cone over it. 
Let $X$ be 
the blowing up of $C_Z$ at its vertex. Notice that $X$ is a $\P^1$-bundle over the strict transform of $Z$. The variety $X$ has canonical divisor equal to $K_X = -H + aE$ where $E$ is the exceptional divisor (which is isomorphic to $Z$) and $H$ the pull-back on $X$ of an hyperplane section of $C_Z$. Since $C_Z$ has multiplicity $N+1$ at its vertex, an easy computation shows that $a= -1$. Thus $-K_X$ is big, but $X$ is not rationally connected since $Z$ is not.}
\end{ex}
One way to measure the lack of nefness of a divisor is the non-emptyness of a perturbation of its (stable) base locus. To be precise, let us recall first that 
the stable base locus $\sbs(D)$ of a divisor $D$ is given by the set-theoretic intersection $\cap_{m\geq1} \Bs(mD)$ of the base loci of all its multiples. Then one defines, following \cite{N} (see also \cite{bouck} definition 3.3 and theorem A.1), the non-nef locus of $D$ to be
the set $\nn(D):=\cup_m\sbs(mD+A)$, where $A$ is a fixed ample divisor . One checks that the definition does not depend on the choice of $A$, and that $\nn(D)=\emptyset$ if, and only if, $D$ is nef.  Notice that the  countable union defining $\nn(D)$ is a finite one, if the ring $R(X,D)$ is finitely generated (see \S 2.4 below).
In any  case $\nn(D)$ is not dense, as soon as $D$ is pseudo-effective, and we always have $\nn(D)\subset \sbs(D)$.
We prove the following.
\begin{thm}\label{thm:main}
Let $(X,\D)$ be a pair such that $-(K_X+\D)$ is big. Then $X$ is rationally connected modulo  an irreducible component  of $\nn(-(K_X+\D))\cup \Nklt(X,\D)$, i.e. there exists an irreducible component $V$ of $\nn(-(K_X+\D))\cup \Nklt(X,\D)$ such that for any general point $x$  of $X$ there exists a rational curve $R_x$ passing through $x$  and intersecting $V$. 
\end{thm}
If the pair is klt, i.e. if the non-klt locus $\Nklt(X,\D)$ is empty (see \S 2.2 for the definition of the non-klt locus), we immediately deduce the following.
\begin{cor}\label{cor:main}
Let $(X,\D)$ be a klt pair such that $-(K_X+\D)$ is big. Then $X$ is rationally connected modulo an irreducible component of the non-nef locus of $-(K_X+\D)$.
\end{cor}
Notice that if $X$ is as in the Example \ref{ex:cone}
then $\nn(-K_X)=E$.

Another immediate consequence of Theorem \ref{thm:main} is the following generalization of Zhang's result \cite[Theorem 1]{zhang} to arbitrary pairs.
\begin{cor}\label{cor:main2}
Let $(X,\D)$ be a pair such that $-(K_X+\D)$ is big and nef. Then $X$ is rationally connected modulo the non-klt locus $\Nklt(X,\D)$.
\end{cor}
Again by Example \ref{ex:cone}, one can see that the previous statement is optimal. Indeed, take $X$ as in Example \ref{ex:cone}, and $\Delta$ equal to the exceptional divisor $E$. Then $X$ is not rationally connected, but it is rationally connected modulo 
$\Nklt(X,\D)=E$.

It may be interesting to point out that, thanks to the work of Campana \cite{cam-pi1}, from  Theorem \ref{thm:main} we also deduce the following. 

\begin{cor}\label{cor:pi1}
Let $(X,\D)$ be a pair such that $-(K_X+\D)$ is big. Let $V$ be an irreducible component  of $\nn(-(K_X+\D))\cup \Nklt(X,\D)$ such that $X$ is rationally connected modulo  $V$. Then, if $V'$ is a desingularization of $V$, the image of $\pi_1(V')$ inside  $\pi_1(X)$ has finite index.
\end{cor}

Theorem \ref{thm:main} is a special case of a more general statement, which may be regarded as a generalization of the main result of Zhang in \cite{zhang2}. Namely we have the following.
\begin{thm}\label{thm:gen}
Let $(X,\D)$ be a pair such that $-(K_X+\D)$ is pseff. 
Let $f:X\dra Z$ be a dominant rational map  with connected fibers.
Suppose that
\begin{equation}\label{eq:assumption}
\textrm {the restriction of  $f$ to $\nn(-(K_X+\D))\cup \Nklt(X,\D)$ does not dominate $Z.$}
\end{equation} 
Then, either 
(i) the variety $Z$ is uniruled; or
(ii) $\kappa(Z)=0$.
Moreover, in the case  $-(K_X+\D)$ is big, under the same assumption (\ref{eq:assumption}) we necessarily have that $Z$ is uniruled.  
\end{thm}
Notice that the possibility (ii) in Theorem \ref{thm:gen} does occur, as one can see by taking $X$ to be the product of a projective space and an abelian variety $Z$. To put Theorem \ref{thm:gen} 
into perspective, notice that in general the image of a variety $X$ with pseff anticanonical divisor may be of general type.  Indeed we have the following instructive example (due to Zhang  \cite[Example, pp. 137--138]{zhang}).

\begin{ex}\label{ex:zhang} Let $C$ be a  smooth curve of arbitrary genus $g$. Let $A$ be an ample line bundle with $\deg(A)>2\deg(K_C)$. Consider the surface 
$S:=\P(\cO_C(A)\oplus\cO_C)$, together with the natural projection $\pi:S\to C$. Then : (i) $-K_S$ is big, but not nef; (ii) for any integer $m>0$ the linear system $|-mK_S|$ contains a fixed component dominating the base $C$. 
\end{ex}

In the smooth case the proof  of Theorem \ref{thm:gen} yields a better result, which is optimal by Example \ref{ex:zhang}.
\begin{thm}\label{thm:genlisse}
Let $X$ be a smooth projective variety  such that $-K_X$ is pseff. 
Let $f:X\dra Z$ be a dominant rational map  with connected fibers. 
Suppose that
\begin{equation}\label{eq:smoothassumption}
\textrm {the restriction of  $f$ to $\Cosupp(\cJ(||-(K_X+\D)||))$ does not dominate $Z.$}
\end{equation} 
Then, either 
(i) the variety $Z$ is uniruled; or
(ii) $\kappa(Z)=0$.
Moreover, in the case  $-(K_X+\D)$ is big, under the same assumption (\ref{eq:smoothassumption}) we necessarily have that $Z$ is uniruled.  
\end{thm}
As a by-product of the previous result we obtain, in the smooth case, the following improvement of Theorem
\ref{thm:main}.
\begin{thm}\label{thm:lisse}
Let $X$ be a smooth projective variety such that $-K_X$ is big. Then $X$ is rationally connected modulo the  locus $\Cosupp(\cJ (X,||-(K_X)||))$.
\end{thm}
Again, notice that if $X$ is as in the Example \ref{ex:cone} then $E=\nn(-K_X)$ is equal to $\Cosupp(\cJ (X,||-(K_X)||))$. 

Moreover,  as remarked by Zhang in \cite{zhang2} and \cite{zhang}, Theorems \ref{thm:gen} and  \ref{thm:genlisse} allow to obtain the following informations on the geometry of the Albanese map, which may be seen as generalizations of \cite[Corollary 2]{zhang2} and \cite[Corollary 3]{zhang}. 
\begin{cor}\label{cor:alb}
Let $(X,\D)$ be a pair such that $-(K_X+\D)$ is pseff. 
Let $\Alb_X:X\dra \Alb(X)$ be the Albanese map (from any smooth model of $X$). 
Suppose that
$\nn(-(K_X+\D))\cup \Nklt(X,\D)$ does not dominate $\Alb(X)$. Then the Albanese map is dominant with connected fibers.  
\end{cor}
\begin{cor}\label{cor:alblisse}
Let $X$ be a smooth projective variety such that $-K_X$ is pseff. 
Let $\Alb_X:X\to \Alb(X)$ be the Albanese map. Suppose that
$\Cosupp(\cJ (X,||-(K_X)||))$ does not dominate $\Alb(X)$. Then the Albanese map is surjective with connected fibers.  
\end{cor}

Our approach to prove Theorem \ref{thm:gen} is similar to  \cite{zhang2}, \cite{zhang} and to \cite{HM},
 and goes as follows: using the hypotheses and multiplier ideal techniques we prove that there exists 
 an effective divisor $L\sim_\Q-(K_X+\D)$ such that the restriction $(\D+L)_{|X_z}$ to the general fiber of $f$ is klt (if $-(K_X+\D)$ is not big we actually add a small ample to it). Then, supposing for simplicity that $X$ and $Y$ are smooth, and $f$ regular, we invoke Campana's positivity result \cite[Theorem 4.13]{ca} to deduce that $f_*({K_X/Z}+\D+L)\sim_\Q -K_Z$ is weakly positive. Then by results contained in \cite{MM}, \cite{bdpp} and \cite{N} the variety $Z$ is forced to be either uniruled or to have zero Kodaira dimension. The technical difficulties come from the singularities of the pair, and from the indeterminacies of the map, but they may be overcome by working on an appropriate resolution. 
In particular, the detailed and very nice account   given by Bonavero in \cite[\S 7.7]{B} on how to reprove \cite{zhang} using the methods of \cite{HM} was of great help for us.  

{\bf Acknowledgements.} We thank Shigeharu Takayama for bringing to our attention the problem from which this work arose. 

%
\section{Preliminaries}\label{S:prel}
%

We recall some  basic definitions and results.
%
\subsection{Notation and conventions}
%
We work over the field of complex numbers.
Unless otherwise specified, a divisor will be integral and Cartier.  
If $D$ and $D'$ are $\Q$-divisors on a projective 
variety $X$ we write $D\sim_\Q D'$, and say that 
$D$ and $D'$ are $\Q$-linearly equivalent, if an integral non-zero multiple of $D-D'$ is linearly equivalent to zero. A divisor $D$ is {\it big} if a large multiple $mD$ induces a birational map $\phi_{mD}: X\dra \P H^0(X,\cO_X(mD))^*$. A divisor $D$ is {\it pseudo-effective} (or {\it pseff}) if its class belongs to the closure of the effective cone  $\overline{\Eff(X)}\subset N^1(X)$. For the basic properties of big and pseff divisors and of their cones see \cite[\S 2.2]{l1}. A variety $X$ is {\it uniruled} if there exists a dominant rational map 
$Y\times \P^1\dra X$ where $Y$ is a variety of dimension $\dim(Y)=\dim(X)-1$. As a combination of 
the deep results contained in \cite{MM} and \cite{bdpp} we have the following characterization of uniruledness : a proper algebraic variety $X$ is uniruled if, and only if, there exists a smooth projective variety $X'$ birational to $X$ whose canonical divisor $K_{X'}$ is not pseff.

%
\subsection{Pairs}
%
We need to recall some terminology about pairs.  Let $X$ a normal projective variety, $\D = \sum a_i D_i$ an effective Weil $\Q$-divisor. If $K_X + \D$ is $\Q$-Cartier, we say that $(X,D)$ is a pair. Consider a log-resolution $f : Y \rightarrow X$ of the pair $(X,\D)$, that is a proper birational morphism such that $Y$ is smooth and $f^* \D + E$ is a simple normal crossing divisor, where $E$ is the exceptional divisor of $f$. There is a uniquely defined exceptional divisor $\sum_j b_j E_j$ on $Y$ such that
$$K_Y = f^*(K_X + \D) - (f^{-1})_* \D + \sum_j b_j E_j,$$
where $(f^{-1})_* \D$ denotes the strict transform of $\D$ in $Y$.

We say that $(X,\D)$ is log-terminal or Kawamata log-terminal (klt for short) if for each $i$ and $j$, we have $a_i < 1$ and $b_j > -1$. If the previous inequalities are large, we say that the pair is log-canonical. We say that the pair $(X,\D)$ is klt or log-canonical at a point $x$ if these inequalities are verified for each of the divisors $D_i, E_j$ having a non empty intersection with $f^{-1}(\{x\})$. We call $\Nklt(X,\D)$ the set of points $x$ at which the pair $(X,\D)$ is not klt.
We refer the reader to \cite{ko} for a detailed treatment of this subject.
%
\subsection{Multiplier ideals}
%
We will freely use the language of multiplier ideals as described in \cite{l2}. We recall in particular that we can associate to a complete linear serie $|D|$ and a rational number $c$, an asymptotic multiplier ideal ${\mathcal J}(c \cdot \| D \| )$ defined as the unique maximal element for the inclusion in
$$\{ {\mathcal J} (\frac{c}{k} \cdot {\frak b}(|kD|) \}_{k \geqslant 1}.$$
For any integer $k\ge 1$, there is an inclusion ${\frak b}(|kD|) \subset {\mathcal J}(k \cdot \| D \| )$. One important property is that, when the divisor $D$ is big, the difference between the two ideals above is uniformly bounded in $k$ (see \cite[Theorem 11.2.21]{l2}  for the precise statement).
%
\subsection{Non-nef locus}
%

The non-nef locus (also called the restricted base locus) of  a divisor $L$ depend only on the numerical equivalence class of $L$ (see \cite{Na}). 
Moreover, by \cite[Corollary 2.10]{elmnp1} if $X$ is smooth 
we have that 
$$
\nn(L)=\cup_{m\in \N} \Cosupp(\cJ(X,||mL||)).
$$
As mentioned in the introduction $L$ is nef if, and only if, $\nn(L)=\emptyset$ and moreover $L$ is pseudo-effective if, and only if, $\nn(L)\not=X$.
For the proofs of the properties mentioned above and more details on non-nef and  stable base loci, see \cite[\S 1]{elmnp1} and  \cite[III, \S1]{N}. Note that though there is no known example, it is expected that $\nn(L)$ is not Zariski-closed in general. Nevertheless, when the algebra $R(X,D)$ is finitely generated, $\nn(L) = {\frak b}(|kD|)$ for a sufficiently large integer $k$ : apply \cite[Example 11.1.3]{l2} and \cite[Corollary 2.10]{elmnp1}.

\subsection{Positivity of direct images}
 
Recall the following positivity result for direct images, due to Campana \cite[Theorem 4.13]{ca}, which improves on previous results obtained by Kawamata \cite{kawab}, Koll\'ar \cite{higherkol} and Viehweg \cite{vie}. We refer the reader to \cite[Ch. 2]{viemod} for the definition and a detailed discussion of the notion of weak positivity, which will not be directly used in this paper.

\begin{thm}[Campana]\label{thm:campana}
Let $f:V'\to V$ be a morphism with connected fibres between smooth projective varieties. Let $\D$ be an
effective 
$\Q$-divisor on $V'$ whose restriction to the generic
fibre is log-canonical. Then, the sheaf 
$$
 f_* \cO_{V'} (m(K_{V'/V}+\D))
$$
is weakly positive for all positive integer $m$ such that $m\D$ is integral.
\end{thm}

We will rather use the following consequence of Campana's positivity result (for a proof see e.g. \cite[Section \S 6.2.1]{D}).
\begin{cor}\label{cor:campana}
Let $f:V'\to V$ be a morphism with connected fibres between smooth projective varieties. Let $\D$ be an
effective 
$\Q$-divisor on $V'$ whose restriction to the generic
fibre is log-canonical. Let $W$ be a general fiber of $f$ and suppose that 
$$
 \kappa(W, m(K_{V'/V}+\D)_{|W})\geq 0.
$$
Then, for every ample divisor $H$ on $V$ there exists a positive integer $b>0$ such that
$$
 h^0 (V', b(K_{V'/V}+\D +f^*H))\not =0.
$$
\end{cor}

%
\section{Proof of the results}\label{S:main}
%

%
\subsection{The proof of the main result}\label{SS:main}
%

We will use the following standard result. 

\begin{lem}\label{lem:kodaira}
Let $(X,\D)$ be a pair, $T$ a smooth quasi-projective variety, $f:X\to T$ a surjective morphism, and $X_t$ the general fiber of $f$.  Suppose that $\D$ is big and that $(X_t,\D_{|X_t})$ is klt. Then there exist two $\Q$-divisors $A$ and $B$, with $A$ ample and $B$ effective such that $\D\sim_\Q A+B$, and the restrictions $(X_t,(A+B)_{|X_t})$ and $(X_t,B_{|X_b})$ are klt.
\end{lem}

Another easy fact which we will use is the following.
\begin{lem}\label{lem:algebre}
 Let $\mathfrak a$, $\mathfrak b$ two ideal sheafs of $X$. Let $f : Y \rightarrow X$ a proper birational morphism such that $f^{-1}{\mathfrak a}$ and $f^{-1}{\mathfrak b}$ are invertible. Let $E$ and $F$ two effective divisors on $Y$ such that ${\mathcal O}_Y (-F) = f^{-1}{\mathfrak a}$ and ${\mathcal O}_Y (-E) = f^{-1}{\mathfrak b}$. Then ${\mathcal O}_Y (-E-F) = f^{-1}({\mathfrak a}\otimes {\mathfrak b})$.
\end{lem}

An important technical ingredient of the proof is provided by the following.
\begin{prop}\label{prop:star}
Let $(X,\D)$ pair and $L$ a big $\Q$-Cartier divisor on $X$. Then there exists an effective $\Q$-Cartier divisor $D\sim_\Q L$ such that $\Nklt (X,\D+D)\subset \nn (L) \cup \Nklt (X,\D)$.
\end{prop}

\begin{proof}

We will construct $D$ as the push-forward of a divisor on a well chosen birational model of $(X,\D)$.

Let $f_Y : Y \rightarrow X$ be a log-resolution of the pair $(X,\D)$. Let $\D_Y$ be an effective $\bf Q$-divisor on $Y$ such that $(f_Y)_*\D_Y = \D$ and
$$ K_Y + \D_Y = f_Y^*(K_X + \D) + F_Y$$
with $F_Y$ effective and not having common components with $\D_Y$. Note that 
\begin{equation}\label{eq:etoile}
f_Y (\Nklt (Y,\D_Y)) \subset \Nklt (X,\D).
\end{equation}
By \cite[Theorem 11.2.21]{l2}, there is an effective divisor $E$ on $Y$ such that for each sufficiently divisible integer $m>0$ we have ${\cJ}(m \cdot\| f_Y^*L \|) \otimes {\cO}_Y(-E) \subset {\frak b}(|m f_Y^*L|)$.
Let $m>0$ be a large and sufficiently divisible integer such that 
\begin{equation}\label{eq:eq1}
\Nklt (Y,\D_Y + \frac{1}{m}E) = \Nklt (Y,\D_Y).
\end{equation} 
Let $f_V : V \rightarrow Y$ be a common log-resolution of ${\cJ}(m \cdot \| f_Y^*L \|) \otimes {\cO}_Y(-E)$ and ${\frak b}(|m f_Y^*L|)$ (see \cite[Definition 9.1.12]{l2}). 
Let $\D_V$ be an effective ${\bf Q}$-divisor such that $(f_V)_*\D_V = \D_Y + \frac{1}{m}E$ and 
$$ K_V + \D_V = f_V^*(K_Y + \D_Y + \frac{1}{m}E) + F_V $$
with $F_V$ effective not having common components with $\D_V$. Again, the pair $(V,\D_V)$ verifies
\begin{equation}\label{eq:1}
f_V (\Nklt (V,\D_V)) \subset \Nklt (Y,\D_Y+ \frac{1}{m}E) = \Nklt (Y,\D_Y).
\end{equation}
Denote by $G_1$ the effective divisor such that
$${\cO}_V(-G_1-f_V^*E) = f_V^{-1} ({\cJ}(m \cdot \| f^*_YL \|)\otimes {\cO}_Y(-E)),$$
and by $G_2$ the divisor such that 
$$
 {\cO}_V(-G_2) = f_V^{-1} {\frak b}(|m f^*_YL|).
$$
Since
${\cJ}(m \cdot \| f^*_YL \|) \otimes {\cO}_X(-E) \subset {\frak b}(|m f^*_Y L|)$, we have $G_1 + f^*_V E \geqslant G_2$. Moreover, by Lemma \ref{lem:algebre}, the divisor $G_1$ verifies 
\begin{equation}\label{eq:G_1}
f_V(G_1) = \Cosupp ({\cJ}(m \cdot \|f^*_Y L \|)).
\end{equation} 
Notice that by construction there exists a base-point-free linear series $|N|$ on $V$ such that $f_V^* f_Y^* mL \sim N + G_2$. Since $|N|$ is base-point-free, there is a ${\bf Q}$-divisor $N'\sim_{{\bf Q}} N$ such that  we have 
\begin{equation}\label{eq:2}
Z:=\Nklt (V,\D_V +  \frac{1}{m}(G_1 + N')) =  \Nklt (V,\D_V + \frac{1}{m}G_1)\subset \Supp G_1 \cup \Nklt (V,\D_V).
\end{equation}
By (\ref{eq:1}), (\ref{eq:G_1})  and (\ref{eq:2})
we get
$$
 f_V(Z) \subset  \Cosupp ({\cJ}(m \cdot \| f^*_YL \|)) \cup \Nklt (Y,\D_Y) \subset \nn (f^*_YL) \cup \Nklt (Y,\D_Y).
$$
Since $G_1 + f_V^*E + N' \geqslant G_2 + N'$, we have
$(f_V)_*(G_1 + f_V^*E + N' ) \geqslant  (f_V)_*(G_2 + N' )$, thus
$$ \Nklt(Y,\D_Y + \frac{1}{m}((f_V)_*(N' + G_2))) \subset \Nklt(Y,\D_Y + \frac{1}{m}(E + (f_V)_*(N' + G_1))).$$
Since  we have
$$\Nklt(Y,\D_Y + \frac{1}{m}(E + (f_V)_*(N' + G_1))) = f_V (\Nklt(V,\D_V + \frac{1}{m}(N' + G_1))),$$
we deduce that 
$$\Nklt(Y,\D_Y + \frac{1}{m}( (f_V)_*(N' + G_2))) \subset \nn (f^*_YL) \cup \Nklt (Y,\D_Y).
$$
Note that $(f_V)_*(N' + G_2)) \sim_\Q mf^*_YL$.
Thanks to (\ref{eq:etoile}) to conclude the proof it is sufficient to   
 prove that  $f_Y(\nn (f^*_YL)) \subset \nn (L)$ and set $D:=(f_Y)_*( \frac{1}{m}( (f_V)_*(N' + G_2)))$.
Let $A$ an ample ${\bf Q}$-divisor on $Y$ and $W$ an irreducible component of $\sbs(f^*_YL + A)$. Then there exists an ample ${\bf Q}$-divisor   $H$ on $X$ such that  $A - f^*_YH$ is ample. Thus 
$$f_Y(W) \subset f_Y(\sbs(f^*_Y(L + H))) \subset \sbs(L + H).$$
This is true for every $A$ and therefore
$$
f_Y(\nn (f^*_YL)) \subset \nn (L).
$$
From the previous inclusions, 
we conclude that 
$$\Nklt (X, \D + \frac{1}{m}(f_Y)_*((f_V)_*(N' + G_2))) \subset \nn (L) \cup \Nklt (X,\D)$$
and setting $D:=(f_Y)_*( \frac{1}{m}( (f_V)_*(N' + G_2)))$ we are done.
\end{proof}

\begin{proof}[Proof of Theorem \ref{thm:gen}]
We give the proof in the case $-(K_X+\D)$ is pseff and point out at the end how to obtain a stronger conclusion when it is big.

We may assume that $Z$ is smooth. 
Fix an ample divisor $H_X$ on $X$, and a rational number $0<\d\ll 1$. From  Proposition \ref{prop:star} 
applied to the big divisor $L=-(K_X+\D)+\d H_X$ we deduce the existence of an effective $\bf Q$-divisor $D\sim_\Q -(K_X+\D)+\d H_X$ such that $\Nklt (X,\D+D)$ does not dominate the base $Y$. From a result on generic restrictions of multiplier ideals \cite[Theorem 9.5.35]{l2}, we have that the restriction to the generic fiber $(X_z,(\D+D)_{|X_z})$ is a klt pair. In conclusion, by the above discussion and 
Lemma \ref{lem:kodaira}, we may replace our original pair $(X,\D)$ by  a new pair, which, by abuse of notation, we will again call $(X,\D)$, such that:
\begin{enumerate}\label{enu:red}
\item[(i)] $K_X+\D\sim_\Q \d H_X$;   
\item[(ii)] $\D=A+B$, where $A$ and $B$ are $\Q$-divisor which are respectively ample and effective, and $A$ may be taken such that $A\sim_\Q \frac{\d}{2}H_X$;
\item[(iii)] the restrictions $(X_z, \D_{|X_z})$ and $(X_z, B_{|X_z})$ to the  fiber  over the general point $z\in Z$ are klt.
\end{enumerate}
Consider a proper birational morphism 
$\mu: Y\to X$ from a smooth projective variety $Y$onto $X$ providing a log-resolution of the pair 
$(X,\D)$ and resolving the indeterminacies of the map $X\dra Z$. In particular we have the following commutative diagram
\begin{equation}
\xymatrix{ Y\ar[r]^{\mu}\ar_{g}[dr]&X\ar@{-->}^{f}[d]\\
 & Z.}
\end{equation}
We write
\begin{equation}\label{eq:can}
 K_Y+\E=\mu^*(K_X+\D)+E,
\end{equation}
where $\E$ and $E$ are effective divisor without common components and $E$ is $\mu$-exceptional.
Notice that from (\ref{eq:can}) and the fact that $E$ is exceptional we deduce that
\begin{equation}\label{eq:koddim}
\kappa(Y,K_Y+\E)=\kappa(X,K_X+\D)\ \  
\end{equation}
whereas from (\ref{eq:can}) together with (iii) above we get that
\begin{equation}\label{eq:restr}
{\textrm{the restriction $(Y_z, \E_{|Y_z})$  to the fiber of $g$  over the general point $z\in Z$ is klt.}}
\end{equation}
More precisely, if we write 
\begin{equation}\label{eq:E}
\E=\mu^*A + B_Y\sim_\Q \frac{\d}{2} \mu^*H_X + B_Y
\end{equation}
we have 
\begin{equation}\label{eq:restr'}
{\textrm{the restriction $(Y_z, (B_Y)_{|Y_z})$  to the fiber of $g$  over the general point $z\in Z$ is klt.}}
\end{equation}
We have the following result.
\begin{lem}\label{lem:elem}
In the above setting, there exists an ample $\Q$-divisor $H_Z$ on $Z$ such that $\mu^*H_X\sim_\Q2g^*H_Z+\G$ where $\G$ is effective and the restriction $(Y_z,(B_Y+\G)_{Y_z})$ to the general fiber is klt.
\end{lem}
\begin{proof}[Proof of Lemma \ref{lem:elem}]
Let $F$ be the exceptional divisor of $\mu$. 
For $0<\e\ll 1$ the divisor $\m^*H_X-\e F$ is ample. 
Take a very ample divisor $H_1$  on $Z$, an integer $m\gg0$ large enough so that $m(\m^*H_X-\e F)-g^*H_1$ is base-point-free and a general member 
$$
 A_m\in |m(\m^*H_X-\e F)-g^*H_1|.
$$
Then $\mu^*H_X\sim_\Q \frac{1}{m}A_m + \e F+\frac{1}{m}g^*H_1$. 
For $m\gg0$ and $0<\e\ll1$, and $z\in Z$ general, the pair 
$$
 (Y_z, (\frac{1}{m}A_m+\e F+ B)_{|Y_z})
$$
is klt, since $ (Y_z, B_{|Y_z})$ is. To conclude the proof it is sufficient to set $H_Z:=\frac{1}{2m}H_1$ and $\G:=\frac{1}{m}A_m + \e F$.
\end{proof}
By Lemma \ref{lem:elem} and (\ref{eq:can}) we have that
$$
 (K_{Y/Z}+B_Y+\G)_{|Y_z}\sim_\Q (K_Y+\E)_{|Y_z}\sim_\Q E_{|Y_z}
$$
is effective. Therefore, by Corollary \ref{cor:campana}, there exists a positive integer $b>0$ such that 
\begin{equation}\label{eq:not0}
h^0(Y, b(K_{Y/Z}+B_Y+\G+g^*(\d H_Z)))\not =0.
\end{equation}
Consider a divisor 
\begin{equation}\label{eq:M}
M\in |b(K_{Y/Z}+B_Y+\G+g^*(\d H_Z))|.
\end{equation}
Let $C\subset Y$ be a mobile curve coming from $X$, i.e. the pull-back via $\mu$ of a general complete intersection curve on $X$. 
Then $C\cdot M\geq 0$. Equivalently we have 
\begin{equation}\label{eq:equiv}
C\cdot g^*K_Z\leq  C\cdot g^*(\d H_Z)+ b(K_Y+B_Y+\G)\cdot C.
\end{equation}
Notice that by (\ref{eq:can}), (\ref{eq:E}) and Lemma \ref{lem:elem} we have
\begin{equation}\label{eq:utile}
 (K_Y+B_Y+\G)\sim_\Q \mu^*(K_X+\D)+E+2\d g^*H_Z.
\end{equation}
Now, 
\begin{enumerate}
\item[(a)] $\mu^*(K_X+\D)\cdot C\to 0$ as $\d\to 0,$ since $K_X+\D\sim_\Q \d H_X$ by (i) above;
\item[(b)] $E\cdot C=0,$ since $E$ is $\mu$-exceptional and $C$ is the pull-back via $\mu$ of a curve on $X$;
\item[(c)]   $2\d (g^*H_Z)\cdot C\to 0$ as $\d\to 0$.
\end{enumerate}
Therefore letting $\d\to 0$ in (\ref{eq:equiv}), from (\ref{eq:utile}) and (a), (b) and (c) we deduce
\begin{equation}\label{eq:leq}
C\cdot g^*K_Z\leq 0.
\end{equation}
If we have strict inequality $C\cdot g^*K_Z< 0$, then $Z$ is uniruled by \cite{MM}. Assume now that $C\cdot g^*K_Z= 0$ and $Z$ is not uniruled. By \cite{bdpp} the canonical class $K_Z$ is pseff 
 and we can therefore consider its divisorial Zariski decomposition $K_Z\equiv P+N$ into a positive part $P$ which is nef in codimension $1$, and a negative one $N$, which is an effective $\R$-divisor (see \cite{bouck} and \cite{N}). 
Since the family of general complete intersection forms a  connecting family (in the sense of \cite[Notation 8.1]{bdpp}) and $g(C)\cdot K_Z=0$, by \cite[Theorem 9.8]{bdpp} we have that $P\equiv 0$. 
Then the numerical Kodaira dimension $\kappa_\sigma(Z)=0$ and hence $\kappa(Z)=0$, by \cite[Chapter V, Corollary 4.9]{N}. This concludes the proof of the theorem when $-(K_X+\D)$ is only assumed to be pseff. 

If moreover $-(K_X+\D)$ is big, then we do not need to add the small ample $\d H_X$ to it in order to apply Proposition \ref{prop:star}. In this case we may  therefore replace our original pair $(X,\D)$ by  a new pair, again called $(X,\D)$, such that:
\begin{enumerate}\label{enu:red}
\item[(j)] $K_X+\D\sim_\Q 0$;   
\item[(jj)] $\D=A+B$, where $A$ and $B$ are $\Q$-divisor which are respectively ample and effective;
\item[(jjj)] the restrictions $(X_z, \D_{|X_z})$ and $(X_z, B_{|X_z})$ to the  fiber  over the general point $z\in Z$ are klt.
\end{enumerate}
Then we argue as in the pseff case. 
Notice that in this case, from (\ref{eq:koddim}) and (j) above we deduce that
\begin{equation}\label{eq:kod0}
\kappa(Y,K_Y+\E)=0.
\end{equation}
As in Lemma \ref{lem:elem}, we write $\mu^*A\sim_\Q 2g^*A_Z+\G$, with $A_Z$ ample on $Z$, $\G$ effective and the restriction $(Y_z,(B_Y+\G)_{Y_z})$ to the general fiber klt. As before, using Campana's positivity result  we deduce
\begin{equation}\label{eq:not0big}
h^0(Y, b(K_{Y/Z}+B_Y+\G+g^*(A_Z)))\not =0.
\end{equation}
Hence for all $m>0$ sufficiently divisible, the group 
$H^0(Z, mb(K_Z+A_Z))$ injects into 
$$H^0(Y, mb(K_{Y/Z}+B_Y+\G+g^*A_Z+f^*(K_Z+A_Z)))=
H^0(Y, mb(K_{Y}+B_Y+\mu^*A))
$$
(for the equality above we use again Lemma \ref{lem:elem}). Since $K_Y+\E=K_{Y}+B_Y+\mu^*A$ 
we deduce that 
\begin{equation}\label{eq:ineq}
0=\kappa (Y, K_Y+\E)\geq \kappa (Z,K_Z+A_Z). 
\end{equation}
If the variety $Z$ were not uniruled, by \cite{bdpp} and \cite{MM} this would imply that 
$K_Z$ is pseff. Therefore $K_Z+A_Z$ is big, and (\ref{eq:ineq}) yields the desired contradiction.
The proof of the theorem  in the case $-(K_X+\D)$ big is now completed. 
\end{proof}

Notice that, when $-(K_X+\D)$ is big, the proof  of Theorem \ref{thm:gen} shows the following statement.
which will be used in the proof of Theorem  \ref{thm:main}.

\begin{thm}\label{thm:kod=0}
Let $(X,\D)$ be a pair such that $\kappa(K_X + \D)=0$, and that $\D$ is big. Let $f : X \dasharrow Z$ be a rational map with connected fibers.
Assume that $\Nklt (X,\D)$ doesn't dominate $Z$,  Then $Z$ is uniruled.
\end{thm}

%
\subsection{The proof of the consequences of the main result}\label{SS:maincons}
%

\begin{proof}[Proof of Theorem \ref{thm:main}]
Let $(X,\Delta)$ be a pair with $-(K_X+\D)$ big. From  Proposition \ref{prop:star} we deduce that there exists an effective divisor $D\sim_{\bf Q} -(K_X+\D)$ such that
$$\Nklt (X,\D + D) \subset \nn (-(K_X+\D))\cup \Nklt(X,\D).$$
Let $\nu : X' \rightarrow X$ a log-resolution of $(X,\Delta+D)$.
We write
\begin{equation}\label{eq:canV}
 K_{X'}+{\D'}=\nu^*(K_X+\D+D)+E,
\end{equation}
where ${\D'}$ and $E$ are effective divisors without common components and $E$ is $\nu$-exceptional.
Notice that from (\ref{eq:canV}) and the fact that $E$ is exceptional we deduce that
\begin{equation}\label{eq:koddimV}
\kappa(X',K_{X'}+{\D'})=\kappa(X,K_X+\D+D)=0\ \  
\end{equation}
Notice also that $\nu(\Nklt (X',{\D'})) \subset \nn (-(K_X+\D))\cup \Nklt(X,\D)$.

Consider now the rational quotient $\r:X'\dra R':=R(X')$.
By the main theorem of \cite{GHS}, the variety $R'$ is not uniruled. Therefore from Theorem \ref{thm:kod=0} we obtain that $\Nklt (X',{\D'})$ dominates $R'$.

Let $x\in X$ be a general point. Let  $y\in X'$ be a general point such that $\nu(y)=x$.
Since by generic smoothness a general fiber of $f$ is smooth, thus rationally connected,  there is a single rational curve $R_y$ passing through $y$ and intersecting $\Nklt (X',{\D'})$. Projecting this curve on $X$, and using the fact that $\nu(\Nklt (X',{\D'})) \subset \nn (-(K_X+\D))\cup \Nklt(X,\D)$, we have the same property for $X$, that is : there exists a rational curve $R_x:=\nu(R_y)$ containing $x$ and intersecting 
an irreducible component of 
 $\nn (-(K_X+\D))\cup \Nklt(X,\D)$  (it is sufficient to take any irreducible component $X'$ of  $\nn (-(K_{X}+\D))\cup \Nklt(X,\D)$ containing the component of $\nu(\Nklt (X',\D'))$ dominating $R'$). 
\end{proof}

\begin{proof}[Proof of Corollary \ref{cor:pi1}]
The statement follows immediately from \cite[Theorem 2.2]{cam-pi1}.
\end{proof}

\begin{proof}[Proof of Theorem \ref{thm:genlisse}]
The structure of the  proof is identical to that in the singular case.
In the smooth case we get a better result because instead of Proposition \ref{prop:star} we may invoke the following result.
\begin{prop}\label{prop:star_lisse}\cite[Example 9.3.57 (i)]{l2}
Let $X$ be a smooth projective variety and $L$ a big $\Q$-Cartier divisor on $X$. Then there exists an effective divisor $D\sim_\Q L$ such that $\Nklt (X,D)\subset \cJ (||L||) $. 
\end{prop}
The conclusion now follows.
\end{proof}

\begin{proof}[Proof of Corollary \ref{cor:alblisse}]
Let $X\buildrel{f}\over{\to}Y\buildrel{g}\over{\to}\Alb(X)$ be the Stein factorization of $\Alb_X$. 
Since an abelian variety does not contain rational curves, from the hypothesis and from Theorem \ref{thm:genlisse} it follows that $\kappa(Y)=0$. From \cite{U} it follows that $\Alb_X(X)$ is an abelian variety. Therefore by \cite{kawab} together with the universality of the Albanese map we have that $\Alb_X$ is surjective and with connected fibers. 
\end{proof}
\begin{proof}[Proof of Corollary \ref{cor:alb}]
It is sufficient to pass to a smooth model $X'\to X$ and to argue as in the proof of Corollary \ref{cor:alblisse}.
\end{proof}

%

\vskip 30pt

\noindent
{\small 
Institut de Recherche Math\'ematique Avanc\'ee\\
Universit\'e L. Pasteur et CNRS\\ 
7, Rue R. Descartes - 67084 Strasbourg Cedex, France \\
E-mail : {\tt broustet@math.u-strasbg.fr}} and {\tt pacienza@math.u-strasbg.fr} 

\end{document}